\documentclass[11pt,english,oneside]{amsart}
\usepackage[T1]{fontenc}
\usepackage[latin9]{inputenc}
\usepackage{geometry}
\usepackage{amssymb}
\usepackage{color}

\usepackage{graphicx,color}
\usepackage{amsmath, amssymb, graphics}

\usepackage{graphicx}

\makeatletter
\numberwithin{equation}{section} 
\numberwithin{figure}{section} 
  \@ifundefined{theoremstyle}{\usepackage{amsthm}}{}
  \theoremstyle{plain}
  \newtheorem{thm}{Theorem}[section]
  \theoremstyle{plain}
  
  \theoremstyle{plain}
  
  \theoremstyle{remark}
  
  \theoremstyle{remark}
  
  \theoremstyle{plain}
  \newtheorem{lem}[thm]{Lemma}

\usepackage{geometry}

\def\bfR#1{{\bf R}^#1}

\def\com#1{ \hbox{#1}}

\def\eop{{\vrule height 6pt width 5pt depth 0pt}}\smallskip
\def\<{{\langle }}
\def\>{{\rangle }}

\makeatother

\usepackage{babel}

\def\bfR#1{{\bf R}^#1}

\def\com#1{ \quad\hbox{#1}\quad}

\def\eop{{\vrule height 6pt width 5pt depth 0pt}}\smallskip
\def\<{{\langle }}
\def\>{{\rangle }}

\makeatother

\begin{document}

\title{Algebraic constant mean curvature surfaces in Euclidean space}

\author{ Oscar M. Perdomo }

\date{\today}

\curraddr{Department of Mathematics\\
Central Connecticut State University\\
New Britain, CT 06050\\
}

\email{ perdomoosm@ccsu.edu}

\begin{abstract}

In this paper we prove that the only algebraic constant mean curvature (cmc) surfaces in $\bfR{3}$ of order less than four are the planes, the spheres and the cylinders. The method used heavily depends on the efficiency of algorithms to compute Groebner Bases and also on the memory capacity of the computer used to do the computations. We will also prove that the problem of finding algebraic constant mean curvature hypersurfaces in the Euclidean space completely reduces to the problem of solving a system of polynomial equations.
\end{abstract}


\maketitle
\section{Introduction}

There is no doubt that level sets of a polynomial function on $\bfR{3}$ are among the easiest surfaces to describe. These surfaces are called algebraic, and the degree of the polynomial defines the order of the algebraic surface. Some examples of surfaces with constant mean curvature (cmc) that are algebraic are planes, spheres, cylinders and Enneper surfaces with order 1, 2, 2 and 9 respectively. The planes and the Enneper surfaces are minimal, that is, they have mean curvature 0. It is known that any algebraic minimal surface in $\bfR{3}$ different from a plane must have order greater than 5, see page 162 of Nitsche book \cite{N}.
Here we will be considering non-minimal algebraic surfaces with cmc. The only known examples are the cylinders and the spheres. In this paper we give a first step toward a proof that these surfaces may be the only ones. We prove that the level set of an irreducible polynomial of degree three cannot be a surface with constant mean curvature. The method used consists in finding equations for the coefficients of the polynomial whose level set have cmc. Theoretically, we can get as many equations as we want but the computations needed to get the equations get harder and harder as the number of equations increases as well as the complexity of the equations obtained. 

The author would like to express his gratitude to Professor William Adkins for giving him references and explanations on the real nullstellensatz theorem.

\section{Main Result}

Let us start this section with the following lemma,

\begin{lem}\label{lemma 1} Let $f:\bfR{n}\to {\bf R}$ be a smooth function, $S=f^{-1}(0)$  and $S_r=\{x\in S:\nabla f(x)\, \ne\, {\bf 0}\}$. The mean curvature of $S_r$ is given by

$$ H(x)= \frac{1}{2 (n-1)|\nabla f|^3}\, (2 |\nabla f|^2\Delta f-\<\nabla |\nabla f|^2,\nabla f\>) $$

\end{lem}
\begin{proof}
Since $-\frac{\nabla f}{|\nabla f|}$ defines a Gauss map on $S_r$, we get that the second fundamental form on $S_r$ is given by

$$\hbox{II}(v,w)=\frac{1}{|\nabla f|} \<D^2 f(v),w\>\com{where $D^2f$ is the Hessian matrix of $f$ }$$

If follows that the mean curvature of $S_r$ at a point $x$ is given by

$$H(x)=\frac{1}{(n-1)|\nabla f|} \sum_{i=1}^{n-1}\<D^2 f(v_i),v_i\>\com{where $\{v_1,\dots, v_{n-1}\}$ is any orthonormal basis of $T_xS_r$ }$$

The lemma follows by using the following two facts: $\Delta f(x)=\<D^2 f(\frac{\nabla f}{|\nabla f|}),\frac{\nabla f}{|\nabla f|}\>+\sum_{i=1}^{n-1}\<D^2 f(v_i),v_i\>$ and $\<\nabla |\nabla f|^2,\nabla f\>=\nabla f\<\nabla f,\nabla f\>=2 \<D^2f\, \nabla f,\nabla f\>$.

\end{proof}

\begin{lem}\label{lemma 2}
 Let $f:\bfR{n}\to {\bf R}$ be an irreducible polynomial. If $f^{-1}(0)$ has constant mean curvature $H$, then there exists a polynomial $p:\bfR{n}\to {\bf R}$  such that

$$(2 |\nabla f|^2\Delta f-\<\nabla |\nabla f|^2,\nabla f\>)^2-4(n-1)^2H^2|\nabla f|^6=p\, f$$

\end{lem}

\begin{proof}

It follows from Lemma \ref{lemma 1} and the following variation of the Real Nullstellensatz Theorem that can be found in page 14 of Milnor's book \cite{M}:

{\it Let $V$ be a real or complex algebraic set defined by a single polynomial equation $f(x) = 0$; with $f$ irreducible. In the real case make the additional hypothesis that $V$ contains a regular point of $f$. Then every polynomial which vanishes on $V$ is a multiple of $f$.}

\end{proof}

As a consequence of the previous lemma we have that the problem of finding algebraic hypersurfaces with constant mean curvature in $\bfR{n}$ reduces to the problem of solving a system of polynomial equations.

\begin{thm}
If $f:\bfR{3}\to {\bf R}$ is an irreducible polynomial, then $f^{-1}(0)$ cannot be an immersed surface with cmc.
\end{thm}

\begin{proof}
Let us argue by contradiction. Let us assume that $S=f^{-1}(0)$ is a complete constant mean curvature surface different from a sphere or a cylinder. By considering a dilation of $S$ instead if needed, without loss of generality, we may assume that the mean curvature of $S$ is 1. By \cite{K}, we get that the Gauss curvature of $S$ must change sign. It follows that there must exist a non constant curve $\alpha$ in $S$ where the Gauss curvature vanishes. We will consider two cases: Case I when there is point in $\alpha$ where the  gradient of $f$ is not the zero vector and Case II when the gradient of $f$ vanishes on $\alpha$.

Let $S_r$ be the set of regular points in $S$, that is

$$S_r=\{x\in S\subset \bfR{3}: \nabla f(x)\, \ne\, (0,0,0)\}$$

Notice that $S_r$ is not empty, otherwise $M$ will be contained in the quadric $\frac{\partial f}{\partial x_1}^{-1}(0)$.

{\bf Case I:} Without loss of generality, we can assume that the origin ${\bf 0}=(0,0,0)$ is in $S_r$ and that the gradient of $f$ at ${\bf 0}$ is the vector $(0,0,1)$; that is, we may assume that $f$ takes the following form

\begin{eqnarray*}
f&=&a_1\text{  } x_1{}^3+a_2\text{  }x_2{}^3+a_3\text{  }x_3{}^3 +a_4\text{  }x_1{}^2 x_2 +a_5\text{  }x_1{}^2 x_3+a_6 x_2{}^2 x_1+a_7\text{
  }x_2{}^2 x_3+a_8\text{  }x_3{}^2 x_1+\\
& &a_9 x_3{}^2 x_2+
a_{10}\text{  }x_1\text{  }x_2\text{  }x_3+b_1 x_1{}^2+b_2 x_2{}^2+b_3 x_3{}^2+b_4 x_1\text{  }x_2+b_5 x_1 x_3 + b_6 x_2\text{  }x_3+x_3
\end{eqnarray*}

Since in this case we are assuming that there exists a regular point in $M$ where the Gauss curvature vanishes, we will also assume that the Gauss curvature at ${\bf 0}$ is $0$, and moreover, that  $(1,0,0)$ defines a principal direction associated with the principal curvature $2$ and that $(0,1,0)$ defines a principal direction associated with the principal curvature $0$. These assumptions imply that

\begin{eqnarray*}\label{set 1}
b_1=1,\quad b_2=0\com{and} b_4=0
\end{eqnarray*}

If we define the $g:\bfR{3}\to {\bf R}$ by

$$  g=2 \nabla f \, \Delta f \, -\, \<\nabla|\nabla f|^2,\nabla f\>\,  -\,  4|\nabla f|^3 $$

then, by Lemma \ref{lemma 1} we get that $g(x)=0$ anytime $f(x)=0$. It follows that the gradient of $g$ must be a multiple of
the gradient of $f$ near the origin. Therefore the functions

$$G_1=\frac{\partial g}{\partial x_1}\frac{\partial f}{\partial x_3}-\frac{\partial g}{\partial x_3}\frac{\partial f}{\partial x_1}\com{and}
G_2=\frac{\partial g}{\partial x_2}\frac{\partial f}{\partial x_3}-\frac{\partial g}{\partial x_3}\frac{\partial f}{\partial x_2}$$

must vanish at every point where $f$ vanishes. Using the same argument we have that the functions

$$G_{11}=\frac{\partial G_1}{\partial x_1}\frac{\partial f}{\partial x_3}-\frac{\partial G_1}{\partial x_3}\frac{\partial f}{\partial x_1},\,
G_{12}=\frac{\partial G_1}{\partial x_2}\frac{\partial f}{\partial x_3}-\frac{\partial G_1}{\partial x_3}\frac{\partial f}{\partial x_2}\com{and}
G_{22}=\frac{\partial G_2}{\partial x_2}\frac{\partial f}{\partial x_3}-\frac{\partial G_2}{\partial x_3}\frac{\partial f}{\partial x_2}$$

vanish at every point where $f$ vanishes. Likewise, we can keep using the same argument and define the functions

$$G_{111},\, G_{112},\, G_{122},\, G_{222},\, G_{1111},\, G_{1112},\, G_{1122},\, G_{1222},\, G_{2222}$$

$$G_{11111},\, G_{11112},\, G_{11122},\, G_{11222},\, G_{12222}\com{and} G_{22222}$$

When we evaluate the functions above at {\bf 0} we get polynomial equations on the coefficients of $f$. They are of the form

$$q_{i_1\dots i_k}=0\com{where} q_{i_1\dots i_k}=G_{i_1\dots i_k}(0,0,0)$$

A direct computation shows that $q_1$ and $q_2$ are given by

$$q_1= 4 \left(3 a_1+a_6-3 b_5\right)\com{and}
q_2=4 \left(3 a_2+a_4-b_6\right)$$

Therefore, from the equations $q_1 =0$ and $q_2=0$ we get that
\begin{eqnarray*}\label{set 2}
a_6=-3 a_1+3 b_5\com{and}a_4=-3 a_2+b_6
\end{eqnarray*}
A direct computation shows that

\noindent\(q_{11}\, =\, -8 \left(6 a_5+a_7-3 \left(-2+2 b_3-a_1 b_5+b_5^2+a_2 b_6\right)\right)\)

\noindent\(q_{12}\, =\, -12 \left(a_{10}-2 \left(a_2 b_5+\left(a_1-b_5\right) b_6\right)\right)\)

\noindent\(q_{22}\, =\, -8 \left(a_7+3 \left(-a_1 b_5+b_5^2+a_2 b_6\right)\right)\)

Therefore, from the equations $q_{11} =0$, $q_{12}$ and $q_{22}=0$ we get that
\begin{eqnarray*}\label{set 3}
a_5&=& -1+b_3-a_1 b_5+b_5^2+a_2 b_6\\
a_{10} &=& 2 \left(a_2 b_5+\left(a_1-b_5\right) b_6\right)\\
a_7&=& 3 a_1 b_5-3 \left(b_5^2+a_2 b_6\right)
\end{eqnarray*}

A direct computation shows that

\noindent\(q_{111}=24 \left(10 a_8+4 a_1^2 b_5+b_5 \left(13+6 a_2^2-7 b_3+4 b_5^2+a_2 b_6-3 b_6^2\right)+a_1 \left(-23+7 b_3-8 b_5^2+2 a_2 b_6+3 b_6^2\right)\right)\)

\noindent\(q_{112}=-24 \left(-2 a_9+4 a_2^2 b_6+\left(2-2 a_1^2+a_1 b_5+b_5^2\right) b_6+a_2 \left(-11+5 b_3-6 a_1 b_5+7 b_5^2+4 b_6^2\right)\right)\)

\noindent\(q_{122}=24 \left(12 a_1^2 b_5-a_1 \left(-3+3 b_3+24 b_5^2+10 a_2 b_6+3 b_6^2\right)+b_5 \left(2 a_2^2+7 a_2 b_6+3 \left(-1+b_3+4 b_5^2+b_6^2\right)\right)\right)\)

\noindent\(q_{222}=24 \left(20 a_2^2 b_6+3 \left(2 a_1^2-5 a_1 b_5+3 b_5^2\right) b_6+a_2 \left(1+b_3-14 a_1 b_5+15 b_5^2+4 b_6^2\right)\right)\)

Therefore, from the equations $q_{111} =0$ and $q_{112}=0$ we get that

\begin{eqnarray*}\label{set 4}
a_8&=&\frac{1}{10} \big{(}-4 a_1^2 b_5+a_1 \left(23-7 b_3+8 b_5^2-2 a_2 b_6-3 b_6^2\right)-b_5 \left(13+6 a_2^2-7 b_3+4 b_5^2+a_2 b_6-3 b_6^2\right)\big{)}\\
a_9&=&\frac{1}{2} \left(4 a_2^2 b_6+\left(2-2 a_1^2+a_1 b_5+b_5^2\right) b_6+a_2 \left(-11+5 b_3-6 a_1 b_5+7 b_5^2+4 b_6^2\right)\right)
\end{eqnarray*}

A direct computation shows that
\begin{eqnarray*}
q_{1111}&=& -\frac{48}{5} \big{(}96 a_1^3 b_5+120 a_2^3 b_6+3 a_1^2 \left(-19+26 b_3-86 b_5^2-4 a_2 b_6-6 b_6^2\right)+\\
& &3 a_1 b_5 \left(73-12 a_2^2-47b_3+76 b_5^2-38 a_2 b_6+12 b_6^2\right)+ \\
& &a_2 b_6 \left(-295+85 b_3+156 b_5^2+30 b_6^2\right)+a_2^2 \left(75+60 b_3+56 b_5^2+110 b_6^2\right)+\\
& & 3 \left(50+50 a_3-54 b_5^2-22 b_5^4+b_3 \left(-50+21 b_5^2\right)-6 b_5^2 b_6^2\right)\big{)}\\
q_{1112}&=&-\frac{24}{5} \big{(}10 a_1^3 b_6-3 a_1^2 b_5 \left(26 a_2+55 b_6\right)+\\
& &b_5\left(-162 a_2^3+223 a_2^2 b_6+a_2 \left(-421+169 b_3+67 b_5^2-139
b_6^2\right)-5 b_6 \left(-27+15 b_3+35 b_5^2+6 b_6^2\right)\right)+\\
& &a_1 \left(-74 a_2^2 b_6+15 b_6 \left(-9+5 b_3+22 b_5^2+2 b_6^2\right)+a_2 \left(496-214
b_3+b_5^2+129 b_6^2\right)\right)\big{)}\\
q_{1122}&=&\frac{24}{5} \big{(}48 a_1^3 b_5-160 a_2^3 b_6-2 a_2^2 \left(-85+70 b_3+181 b_5^2-70 b_6^2\right)+\\
& &2 a_1^2 \left(-183+72 b_3+108 b_5^2+52
a_2 b_6-12 b_6^2\right)+\\
& &3 b_5^2 \left(-117+63 b_3+104 b_5^2+17 b_6^2\right)+3 a_2 b_6 \left(-30+30 b_3+221 b_5^2+20 b_6^2\right)-\\
& &a_1 b_5 \left(-717-312 a_2^2+333 b_3+576 b_5^2+782 a_2 b_6+27 b_6^2\right)\big{)}\\
q_{1222}&=&\frac{24}{5} \big{(}90 a_1^3 b_6+3 a_1^2 b_5 \left(86 a_2+45 b_6\right)+\\
& &b_5 \left(-18 a_2^3+467 a_2^2 b_6+15 b_6 \left(-9+9 b_3+21 b_5^2+8
b_6^2\right)+a_2 \left(-339+231 b_3+513 b_5^2+259 b_6^2\right)\right)-\\
& &a_1 \left(186 a_2^2 b_6+15 b_6 \left(-9+9 b_3+36 b_5^2+8 b_6^2\right)+ a_2 \left(-384+246
b_3+771 b_5^2+379 b_6^2\right)\right)\big{)}\\
q_{2222}&=&48 \big{(}8 a_2^3 b_6+3 a_1 b_5 \left(18-4 a_2^2-12 b_3+12 b_5^2-25 b_6^2\right)+\\
& &9 b_5^2 \left(-3+2 b_3-2 b_5^2+5 b_6^2\right)+a_2 b_6
\left(9+5 b_3+7 b_5^2+18 b_6^2\right)+\\
& &a_1^2 \left(-27+18 b_3-18 b_5^2-4 a_2 b_6+30 b_6^2\right)+a_2^2 \left(-31+28 b_3+22 b_5^2+66 b_6^2\right)\big{)}\\
\end{eqnarray*}

By a direct computation we also get that

\begin{eqnarray*}
q_{11111}&=&24 \big{(}-6 a_1^3 b_5 \left(82 a_2-195 b_6\right)-180 a_1^4 b_6+600 a_2^4 b_6+\\
& &2 a_2^3 \left(-1095+375 b_3+529 b_5^2-420 b_6^2\right)+45
b_5^2 b_6 \left(-14+8 b_3-34 b_5^2+13 b_6^2\right)+\\
& &2 a_2^2 b_6 \left(-345+150 b_3-891 b_5^2+50 b_6^2\right)+\\
& &2 a_1^2 \left(42 a_2^2 b_6+a_2 \left(-903+267
b_3+303 b_5^2-182 b_6^2\right)+15 b_6 \left(-21+12 b_3-111 b_5^2+13 b_6^2\right)\right)+\\
& &a_2 \left(-438 b_5^4+3 b_5^2 \left(-542+58 b_3-583 b_6^2\right)+5
b_6^2 \left(37+13 b_3+44 b_6^2\right)\right)+\\
& &a_1 b_5 \left(-828 a_2^3+1858 a_2^2 b_6-15 b_6 \left(-84+48 b_3-258 b_5^2+65 b_6^2\right)+a_2 \left(3432-708
b_3+324 b_5^2+2138 b_6^2\right)\right)\big{)}\\
q_{11112}&=&-\frac{72}{5} \big{(}72 a_1^4 b_5+2 a_1^3 \left(-477+63 b_3-744 b_5^2-82 a_2 b_6-183 b_6^2\right)+\\
& &2 a_1^2 b_5 \left(1056-250 a_2^2+51 b_3+2016
b_5^2+621 a_2 b_6+404 b_6^2\right)+\\
& &b_5 \big{(}-12 a_2^4+1272 b_5^4-78 a_2^3 b_6-25 b_6^2 \left(-11+11 b_3+9 b_6^2\right)+\\
& &b_5^2 \left(204+354 b_3+76
b_6^2\right)-\\
& &a_2 b_6 \left(-902+498 b_3+476 b_5^2+167 b_6^2\right)-2 a_2^2 \left(-732+253 b_3+150 b_5^2+745 b_6^2\right)\big{)}+\\
& &a_1 \big{(}-3888 b_5^4+396
a_2^3 b_6+25 b_6^2 \left(-11+11 b_3+9 b_6^2\right)+2 a_2^2 \left(-797+243 b_3+410 b_5^2+111 b_6^2\right)-\\
& &2 b_5^2 \left(681+291 b_3+259 b_6^2\right)+2
a_2 b_6 \left(274 b_3-7 \left(78+43 b_5^2-28 b_6^2\right)\right)\big{)}\big{)}\\
q_{11122}&=&-\frac{24}{5} \big{(}120 a_2^4 b_6-a_2^2 b_6 \left(2140+1020 a_1^2-1960 b_3+2720 a_1 b_5-4211 b_5^2+940 b_6^2\right)+\\
& &2 a_2^3 \left(-435+75 b_3+90 a_1 b_5-134 b_5^2+1420 b_6^2\right)+\\
& &3 \left(a_1-b_5\right) b_6 \big{(}180 a_1^3-672 a_1^2 b_5-3 a_1 \left(-418+122 b_3+67 b_5^2-42 b_6^2\right)+\\
& &b_5 \left(-1179+531 b_3+693 b_5^2+124 b_6^2\right)\big{)}+\\
& &a_2 \big{(}1860 a_1^3 b_5-2 a_1^2 \left(-525+465 b_3+2916 b_5^2+844 b_6^2\right)+a_1 b_5 \left(-3006+2364
b_3+6429 b_5^2+8525 b_6^2\right)+\\
& &3 \left(30+60 a_3-10 b_3^2+637 b_5^2-819 b_5^4+350 b_6^2-2249 b_5^2 b_6^2-220 b_6^4-b_3 \left(20+473 b_5^2+350 b_6^2\right)\right)\big{)}\big{)}\\
\end{eqnarray*}

A direct computation using a computer program (Mathematica for example) shows that a Groebner Basis for the set of polynomials

$$\{q_{122},q_{222},q_{1111},q_{1112},q_{1122},q_{1222},q_{2222},q_{11111},q_{11112},q_{11122}\}$$

with respect to the variables $\left\{a_1,a_2,b_5,b_6,b_3\right\}$ is \(\left\{-1-a_3+b_3,a_2,a_1-b_5\right\}\). Therefore it follows that

$$a_2=0,\quad b_5=a_1\com{and} b_3=1+a_3$$

Replacing these and the previous equations for the coefficients of $f$ we get that

$$f=\left(1+a_1 x_1+b_6 x_2+a_3 x_3\right) \left(x_1^2+x_3+x_3^2\right)$$

which is a contradiction because $f$ is irreducible. This finishes the proof of Case I.

Let us prove case II. Let us assume that the curve $\alpha$ defined above contains the origin, it is parametrized by arc-length and its velocity at the origin is the vector $(1,0,0)$. The assumption made in case II implies that not only $f$ vanishes on the points in $\alpha$ but all the 3 partial derivatives of $f$ vanish on these points. Since we are assuming that the origin is in $S$ then,

\begin{eqnarray*}
 f&=&a_1 x_1^3 + a_2 x_2^3 + a_3 x_3^3 + a_4 x_1^2x_2 + a_5 x_1^2x_3 +
  a_6x_2^2x_1 + a_7  x_2^2 x_3 + a_8 x_3^2 x_1 + \\
 & &a_9 x_3^2 x_2 + a_{10} x_1 x_2 x_3 + b_1 x_1^2 + b_2 x_2^2 + b_3 x_3^2 + b_4 x_1 x_2 +
  b_5 x_1 x_3 + b_6 x_2 x_3
\end{eqnarray*}

Since the functions $t\to \frac{\partial f}{\partial x_i}(\alpha(t))$ vanish, then their derivatives at zero are zero, that is,
$\frac{\partial^2 f}{\partial x_i\partial x_1}(0)=0$; therefore

$$b_{11}=0,\quad b_4=0 \com{and} b_5=0$$

A direct computation using the fact that $\<\alpha^{\prime\prime}(0),\alpha^{\prime}(0)\>=0$, shows that the second derivative of the function $t\to \frac{\partial f}{\partial x_1}(\alpha(t))$ at zero is $\frac{\partial^3 f}{\partial x_1^3}(0)$; therefore, we also have that

$$a_1=0$$

These equations on the coefficients of $f$ imply that $f(x_1,0,0)=0$ for all $x_1\in {\bf R}$, that is, we have that the $x_1$-axis is contained in $S$. Let us define

$$u=(2 |\nabla f|^2\Delta f-\<\nabla |\nabla f|^2,\nabla f\>)^2-4(n-1)^2 |\nabla f|^6$$

by either Lemma \ref{lemma 1} or Lemma \ref{lemma 2}, we have that the polynomial in the variable $x_1$,
$p_1(x_1)=u(x_1,0,0)$ vanishes. A direct computation shows that the coefficient of $x_1^{12}$ is $-16 (a_4^2 + a_5^2)^3$; therefore,

$$a_4=0\com{and} a_5=0$$

If we assume that $a_8\ne 0$, then, a direct computation shows that $f(-\frac{b_3+a_3 x_3}{a_8},0,x_3)=0$ for all $x_3$; therefore, we would have that $p_2(x_3)=u(-\frac{b_3+a_3 x_3}{a_8},0,x_3)$ must also vanish.  A direct computation shows that the coefficient of $x_3^{12}$ is $-\frac{16((a10 a3 - a8 a9)^2 + a8^2 (a3^2 + a8^2))^3}{a_8^6}$; therefore, $a_8$ must be zero which contradict the initial assumption that $a_8\ne0$. Therefore

$$a_8=0$$

If we assume that $a_6\ne 0$, then, a direct computation shows that $f(-\frac{b2 + a2 x2}{a6},x_2,0)=0$ for all $x_2$; therefore, we would have that $p_3(x_2)=u(-\frac{b2 + a2 x2}{a6},x_2,0)$ must also vanish.  A direct computation shows that the coefficient of $x_2^{12}$ is $-\frac{16((a10 a2 - a6 a7)^2 + a6^2 (a2^2 + a6^2))^3}{a_6^6}$; therefore, $a_6$ must be zero which contradict the initial assumption that $a_6\ne0$. Therefore

$$a_6=0$$

At this point we may assume that $a_{10}$ is not zero, otherwise $f$ would be independent of $x_1$ which is impossible since the only cylinder that has mean curvature 1 is a circular cylinder. Without loss of generality we may assume that

$$a_{10}=1$$

A direct computation shows that for any nonzero $x_2$ and $x_3$

$$f\big{(}\,\frac{-b_2 x_2^2 - a_2 x_2^3 - b_6 x_2 x_3 - a_7 x_2^2 x_3 - b_3 x_3^2 - a_9 x_2 x_3^2 -
 a_3 x_3^3}{ x_2 x_3}\, ,\, x_2\, ,\, x_3\, \big{)}=0$$

 Therefore the polynomial function

 $$p_4(x_2,x_3)=(x_2x_3)^6\, u\big{(}\,\frac{-b_2 x_2^2 - a_2 x_2^3 - b_6 x_2 x_3 - a_7 x_2^2 x_3 - b_3 x_3^2 - a_9 x_2 x_3^2 -
 a_3 x_3^3}{ x_2 x_3}\, ,\, x_2\, ,\, x_3\, \big{)}$$

 must also vanish. A direct computation shows that the coefficient of the term $x_3^{18}$ of $p_4$ is $-16b_3^6$; therefore,

 $$b_3=0$$

 A direct computation shows that the coefficient of the term $x_3^{24}$ of $p_4$ is $-16a_3^6$; therefore,

 $$a_3=0$$

Therefore the $f$ takes the following form: $x_2 (b_2 x_2 + a_2 x_2^2 + b_6 x_3 + x_1 x_3 + a_7 x_2 x_3 + a_9 x_3^2)$. This is a contradiction because we are assuming that $f$ is irreducible. This finishes the proof of the theorem. \eop
\end{proof}

\end{document}